\documentclass[a4paper]{article}

\usepackage{amsthm}
\usepackage{amsmath}
\usepackage{amssymb}
\usepackage{graphicx}
\usepackage{color}
\usepackage{ifpdf}
\usepackage{url}
\usepackage{hyperref}
\usepackage{algorithm}
\usepackage{abstract}
\usepackage[usenames,dvipsnames]{xcolor}
\usepackage{xfrac}
\usepackage{comment}

\usepackage[ruled,vlined, linesnumbered, algo2e]{algorithm2e}
\usepackage{amssymb,makeidx,mathrsfs}
\usepackage{fixltx2e}

\usepackage[T1]{fontenc} 
\usepackage[hmarginratio=1:1,top=32mm,columnsep=20pt]{geometry} 
\usepackage{multicol} 
\usepackage[hang, small,labelfont=bf,up,textfont=it,up]{caption} 
\usepackage{booktabs} 
\usepackage{float} 
\usepackage{hyperref} 

\usepackage{abstract} 

\theoremstyle{plain}
\newtheorem{theorem}{Theorem}[section]

\newtheorem{lemma}{Lemma}[section]

\theoremstyle{definition}

\newtheorem{assumption}{Assumption}[section]

\usepackage{fancyhdr} 
\makeatletter
\renewcommand\@biblabel[1]{#1.}
\makeatother

\newcommand{\R}{\mathbb{R}}
\newcommand{\Rext}{\R\cup\{+\infty\}}

\newcommand{\set}[1]{\left\{#1\right\}}
\newcommand{\sets}[1]{\{#1\}}
\newcommand{\norm}[1]{\left\Vert#1\right\Vert}
\newcommand{\norms}[1]{\Vert#1\Vert}

\newcommand{\prox}{\mathrm{prox}}
\newcommand{\tprox}[2]{\mathrm{prox}_{#1}\left(#2\right)}

\newcommand{\dom}[1]{\mathrm{dom}(#1)}

\newcommand{\zero}[1]{{\boldsymbol{0}}}

\newcommand{\Bc}{\mathcal{B}}

\newcommand{\Fc}{\mathcal{F}}

\newcommand{\iprod}[1]{\left\langle #1\right\rangle}

\newcommand{\Exp}[1]{\mathbb{E}\left[#1\right]}
\newcommand{\Exps}[2]{\mathbb{E}_{#1}\left[#2\right]}
\newcommand{\Probn}{\mathbb{P}}

\newcommand{\BigO}[1]{\mathcal{O}\left(#1\right)}

\newcommand{\beforesec}{\vspace{-2.5ex}}
\newcommand{\aftersec}{\vspace{-1.5ex}}
\newcommand{\beforesubsec}{\vspace{-2ex}}
\newcommand{\aftersubsec}{\vspace{-1.5ex}}

\usepackage{algorithm,paralist}
\usepackage{algpseudocode}
\usepackage{ulem}

\title{An Optimal Hybrid Variance-Reduced Algorithm for Stochastic Composite Nonconvex Optimization}

\author{Deyi Liu, Lam M. Nguyen, and Quoc Tran-Dinh}

\begin{document}
\maketitle

\maketitle
\begin{abstract}
In this note we propose a \textbf{new variant} of the \textit{hybrid variance-reduced proximal gradient method} in \cite{Tran-Dinh2019a} to solve a common stochastic composite nonconvex optimization problem under standard assumptions.
We simply replace the \textit{independent unbiased estimator} in our hybrid-SARAH estimator introduced in \cite{Tran-Dinh2019a} by the stochastic gradient evaluated at the same sample, leading to the identical \textit{momentum-SARAH estimator} introduced in \cite{Cutkosky2019}.
This allows us to save one stochastic gradient per iteration compared to \cite{Tran-Dinh2019a}, and only requires two samples per iteration.
Our algorithm is very simple and achieves optimal stochastic oracle complexity bound in terms of stochastic gradient evaluations (up to a constant factor).
Our analysis is essentially inspired by \cite{Tran-Dinh2019a}, but we do not use two different step-sizes.
\end{abstract}


    
\beforesec
\section{Problem Statement and Standard Assumptions}\label{sec:intro}
\aftersec
We consider the following stochastic composite and possibly nonconvex optimization problem:
\begin{equation}\label{eq:ncvx_prob}
\min_{x\in\R^p}\Big\{ F(x) := \Exps{\xi}{f_{\xi}(x)} + \psi(x) \Big\},
\end{equation}
where $f_{\xi}(\cdot) : \R^p\times \Omega \to \R$ is a stochastic function defined, such that for each $x\in\R^p$, $f_{\xi}(x)$ is a random variable in a given probability space $(\Omega, \Probn)$, while for each realization $\xi\in\Omega$, $f_{\xi}(\cdot)$ is differentiable on $\R^p$; and $f(x) := \Exps{\xi}{f_{\xi}(x)}$ is the expectation of the random function $f_{\xi}(x)$ over $\xi$ on $\Omega$; $\psi : \R^p\to\Rext$ is a proper, closed, and convex function.


%

Our algorithm developed in this note relies on the following fundamental assumptions:

\begin{assumption}\label{as:A1} The objective functions $f$ and $\psi$ of \eqref{eq:ncvx_prob} satisfies the following conditions:
\begin{itemize}
\item[$\mathrm{(a)}$] (\textbf{Convexity of $\psi$}) $\psi : \R^p\to\Rext$ is proper, closed, and convex. 
In addition, $\dom{F} := \dom{f}\cap\dom{\psi}\neq\emptyset$.

\item[$\mathrm{(b)}$] (\textbf{Boundedness from below}) 
There exists a finite lower bound 
\begin{equation}\label{eq:lower_bound}
F^{\star} := \inf_{x\in\R^p}\Big\{ F(x) := f(x) + \psi(x) \Big\} > -\infty.
\end{equation}

\item[$\mathrm{(c)}$] (\textbf{$L$-average smoothness})
The expectation function $f(\cdot)$ is $L$-smooth on $\dom{F}$, i.e., there exists $L\in (0, +\infty)$ such that
\begin{equation}\label{eq:L_smooth}
\Exps{\xi}{\norm{\nabla{f}_{\xi}(x) - \nabla{f}_{\xi}(y)}^2} \leq L^2\norms{x - y}^2,~~\forall x, y\in\dom{F}.
\end{equation}
%
\item[$\mathrm{(d)}$] (\textbf{Bounded variance})
There exists $\sigma \in [0, \infty)$ such that 
\begin{equation}\label{eq:bounded_variance2}
\Exps{\xi}{\norms{\nabla{f}_{\xi}(x) - \nabla{f}(x)}^2} \leq \sigma^2, ~~~\forall x\in\dom{F}.
\end{equation}
\end{itemize}
\end{assumption}
These assumptions are very standard in stochastic optimization and required for various gradient-based methods.
Unlike  \cite{Cutkosky2019}, we do not impose a bounded gradient assumption, i.e., $\norms{\nabla{f}(x)} \leq G$ for all $x\in\R^p$.
Algorithm~\ref{alg:A1} below has a single loop and achieves optimal oracle complexity bound since it matches the lower bound complexity in \cite{arjevani2019lower} up to a constant factor.


\beforesec
\section{Hybrid Variance-Reduced Proximal Gradient Algorithm}\label{sec:prox_hybrid}
\aftersec
We first propose a new variant of \cite[Algorithm~1]{Tran-Dinh2019a} for solving \eqref{eq:ncvx_prob} and then analyze its convergence and oracle complexity.

\beforesubsec
\subsection{Main result: Algorithm and its convergence}
\aftersubsec
We propose a \textbf{\textit{novel hybrid variance-reduced proximal gradient method}} to solve \eqref{eq:ncvx_prob} under standard assumptions (i.e., Assumption~\ref{as:A1}) as described in Algorithm~\ref{alg:A1}.

\begin{algorithm}[ht!]\caption{(\textbf{Hybrid Variance-Reduced Proximal Gradient Algorithm})}\label{alg:A1}
\normalsize
\begin{algorithmic}[1]
\itemsep=0.2em
   \State{\bfseries Initialization:} An arbitrarily initial point $x_0 \in\dom{F}$.
   \State\hspace{2ex}\label{step:o1} Choose an initial batch size $\tilde{b} \geq 1$, $\beta \in (0, 1)$, and $\eta > 0$ as in Theorem~\ref{thm:converge_rate} below.
   \State\hspace{2ex}\label{step:o2} Generate an unbiased estimator $v_0 := \frac{1}{\tilde{b}}\sum_{\tilde{\xi}_i\in\widetilde{\Bc}}\nabla{f}_{\tilde{\xi}_i}(x_0)$ at $x_0$ using a mini-batch $\widetilde{\Bc}$.
   \State\hspace{2ex}\label{step:o3} Update $x_1 := \tprox{\eta_0\psi}{x_0 - \eta_0v_0}$.
   \State\hspace{0ex}\label{step:o4}{\bfseries For $t := 1,\cdots, T$ do}
   \State\hspace{2ex}\label{step:i1} Generate a proper sample $\xi_t$ (single sample or mini-batch).
   \State\hspace{2ex}\label{step:i2} Evaluate $v_t$ and update
   \begin{equation}\label{eq:prox_hybrid_alg}
	\left\{\begin{array}{ll}
	v_t &:= \nabla{f_{\xi_t}}(x_t) + (1-\beta)\left[ v_{t-1} - \nabla{f_{\xi_t}}(x_{t-1}) \right]\vspace{1ex}\\
	x_{t+1} &:= \prox_{\eta\psi}(x_t - \eta v_t).
	\end{array}\right.
	\end{equation}
   \State\hspace{0ex}{\bfseries EndFor}
   \State\hspace{0ex}\label{step:o5}Choose $\overline{x}_T$ uniformly from $\set{x_0, x_1, \cdots, x_T}$.
\end{algorithmic}
\end{algorithm}

Compared to \cite[Algorithm~1]{Tran-Dinh2019a}, the new algorithm, Algorithm~\ref{alg:A1}, has two major differences.
First, it uses a new estimator $v_t$ adopted from \cite{Cutkosky2019}.
This estimator can also be viewed as a variant of the hybrid SARAH estimator in \cite{Tran-Dinh2019a} by using the same sample $\xi_t$ for $\nabla{f}_{\xi_t}(x_t)$.
That is
\begin{equation*}
\arraycolsep=0.2em
\begin{array}{llcll}
\text{Hybrid SARAH \cite{Tran-Dinh2019a}:} & v^{h}_t &:= & {\color{blue}(1-\beta)[v^{h}_{t-1} + \nabla{f}_{\xi_t}(x_t) - \nabla{f}_{\xi_t}(x_{t-1})]} + \beta {\color{red}\nabla{f}_{\zeta_t}(x_t)}, & {\color{red}\xi_t \neq \zeta_t}, \vspace{1ex}\\
\text{Momentum SARAH \cite{Cutkosky2019}:} & v_t &:= & {\color{blue}(1-\beta)[v_{t-1} + \nabla{f}_{\xi_t}(x_t) - \nabla{f}_{\xi_t}(x_{t-1})]} + \beta {\color{magenta}\nabla{f}_{\zeta_t}(x_t)}, &{\color{blue}\xi_t = \zeta_t}. 
\end{array}
\end{equation*}
Second, it does not require an extra damped step-size $\gamma$ as in \cite{Tran-Dinh2019a}, making Algorithm~\ref{alg:A1} simpler than the one in \cite{Tran-Dinh2019a}.

To analyze Algorithm~\ref{alg:A1}, as usual, we define the following gradient mapping of \eqref{eq:ncvx_prob}:
\begin{equation}\label{eq:grad_mapping}
G_{\eta}(x) := \tfrac{1}{\eta}\left(x - \prox_{\eta\psi}(x - \eta \nabla f(x))\right),
\end{equation}
where $\eta > 0$ is any given step-size.
It is obvious to show that $x^{\star}\in\dom{F}$ is a stationary point of \eqref{eq:ncvx_prob}, i.e., $0 \in \nabla{f}(x^{\star}) + \partial{\psi}(x^{\star})$ if and only if $G_{\eta}(x^{\star}) = 0$.
We will show that for any $\varepsilon > 0$, Algorthm~\ref{alg:A1} can find $\overline{x}_T$ such that $\Exp{\norms{G_{\eta}(\overline{x}_T)}^2} \leq \varepsilon^2$, which means that $\overline{x}_T$ is an $\varepsilon$-approximate stationary point of \eqref{eq:ncvx_prob}, where the expectation is taken over all the present randomness.

The following theorem establishes convergence of Algorithm~\ref{alg:A1} and provides oracle complexity.

\begin{theorem}\label{thm:converge_rate} 
Under Assumption~\ref{as:A1}, suppose that $\eta \in (0, \frac{1}{2L})$ is a given step-size and $0 < \frac{2L^2\eta^2}{1 - L\eta} \leq \beta < 1$.
Let $\set{x_t}_{t=0}^T$ be generated by Algorithm~\ref{alg:A1}.
Then, we have
\begin{equation}\label{eq:thm_converge_rate}
\frac{1}{T+1}\sum_{t=0}^{T}\Exp{\norms{G_{\eta}(x_t)}^2} \leq \frac{2[ F(x_0) - F^{\star}]}{\eta(T+1)}
+  \frac{\Exp{\norms{v_0 - \nabla f(x_0)}^2}}{\beta(T+1)} {~} + {~} 2\beta\sigma^2.
\end{equation}
In particular, if we choose $\eta := \frac{1}{2L(T+1)^{1/3}}$, $\beta := \frac{1}{(T+1)^{2/3}}$, and $\tilde{b} := \left\lceil \frac{(T+1)^{1/3}}{2} \right\rceil \geq 1$, then the output $\overline{x}_T$ of Algorithm~\ref{alg:A1} satisfies
\begin{equation}\label{eq:thm_opt_converge_rate}
\Exp{\norms{G_{\eta}(\overline{x}_T)}^2}
\leq \frac{4L[F(x_0) - F^{\star}] + 4\sigma^2}{(T+1)^{2/3}}. 
\end{equation}
Consequently, for any tolerance $\varepsilon > 0$, the total number of stochastic gradient evaluations in Algorithm~\ref{alg:A1} to achieves $\overline{x}_T$ such that $\Exp{\norms{G_{\eta}(\overline{x}_T)}^2} \leq\varepsilon^2$ is at most $\mathcal{T}_{\nabla{f}} := \left\lceil \frac{\Delta_0^{1/2}}{2\varepsilon} + \frac{2\Delta_0^{3/2}}{\varepsilon^3} \right\rceil$, where $\Delta_0 := 4\left[ L[F(x_0) - F^{\star}] + \sigma^2\right]$.
\end{theorem}

Theorem~\ref{thm:converge_rate} shows that the oracle complexity of Algorithm~\ref{alg:A1} is $\BigO{\frac{\Delta_0^{1/2}}{\varepsilon} + \frac{\Delta_0^{3/2}}{\varepsilon^3}}$ as in \cite{Tran-Dinh2019a}, where $\Delta_0 := 4( L[F(x_0) - F^{\star}] + \sigma^2)$.
This complexity bound in fact matches the lower bound one in \cite{arjevani2019lower} up to a constant factor under the same assumptions as in Assumption~\ref{as:A1}.
Hence, we conclude that Algorithm~\ref{alg:A1} is \textit{optimal}.

\beforesubsec
\subsection{Convergence Analysis}
\aftersubsec
Let us denote by $\Fc_t := \sigma(\xi_0, \xi_1, \cdots, \xi_t)$ the $\sigma$-filed generated by $\sets{\xi_0, \xi_1,  \cdots, \xi_t}$.
We also denote by $\Exp{\cdot}$ the full expectation over the history $\Fc_t$.
The following lemma establishes a key estimate for our convergence analysis.
We emphasize that Lemma~\ref{lm:variance_redu} is self-contained and can be applied to other types of estimators, e.g., Hessian, and other problems.

\begin{lemma}\label{lm:variance_redu}
Let $v_t$ be computed by \eqref{eq:prox_hybrid_alg} for $\beta \in (0, 1)$.
Then, under Assumption~\ref{as:A1}, we have
\begin{equation}\label{eq:variance_redu}
{\color{blue}\Exps{\xi_t}{\norms{v_t - \nabla f(x_t)}^2} \leq (1-\beta)^2\norms{v_{t-1} - \nabla f(x_{t-1})}^2 + 2(1-\beta)^2L^2\norms{x_t - x_{t-1}}^2 + 2\beta^2\sigma^2.}
\end{equation}
Therefore, by induction, we have
\begin{equation}\label{eq:variance_redu_full}
\begin{array}{lcl}
\Exp{\norms{v_t - \nabla f(x_t)}^2} & \leq & (1-\beta)^{2t}\Exp{\norms{v_0 - \nabla f(x_0)}^2}  + 2\beta\sigma^2 \vspace{1ex}\\
&& + {~} 2L^2\sum_{i=0}^{t-1}(1-\beta)^{2(t-i)}\Exp{\norms{x_{i+1} - x_i}^2}.
\end{array}
\end{equation}
\end{lemma}

\begin{proof}
Let us denote
\begin{equation*}
a_t := (1-\beta)\left[\nabla{f_{\xi_t}}(x_t) - \nabla f(x_t) - \nabla{f_{\xi_t}}(x_{t-1})  + \nabla f(x_{t-1})\right] \quad\text{and}\quad b_t := \beta\left[\nabla{f_{\xi_t}}(x_t) - \nabla f(x_t)\right].
\end{equation*}
Since $\Exps{\xi_t}{a_t} = \Exps{\xi_t}{b_t} = 0$, and \eqref{eq:L_smooth}, we can derive \eqref{eq:variance_redu} as follows: 
\begin{equation*}
\arraycolsep=0.2em
\begin{array}{lcl}
\Exps{\xi_t}{\norms{v_t - \nabla f(x_t)}^2}  &= & \Exps{\xi_t}{\norms{\nabla{f_{\xi_t}}(x_t) + (1-\beta)(v_{t-1} - \nabla{f_{\xi_t}}(x_{t-1})) - \nabla f(x_t)}^2} \vspace{1ex}\\
& = & \Exps{\xi_t}{\norms{(1-\beta)[v_{t-1} - \nabla{f}(x_{t-1})] + a_t + b_t }^2} \vspace{1ex}\\
& = &  (1-\beta)^2\norms{v_{t-1} - \nabla{f}(x_{t-1})}^2 + \Exps{\xi_t}{\norms{a_t + b_t}^2} \vspace{1ex}\\
& \leq & (1-\beta)^2\norms{v_{t-1} - \nabla f(x_{t-1})}^2 + 2\Exps{\xi_t}{\norms{a_t}^2} + 2\Exps{\xi_t}{\norms{b_t}^2} \vspace{1ex}\\
& \leq & (1-\beta)^2\norms{v_{t-1} - \nabla f(x_{t-1})}^2 + 2(1-\beta)^2\Exps{\xi_t}{\norms{\nabla{f_{\xi_t}}(x_t) - \nabla{f_{\xi_t}}(x_{t-1})}^2} \vspace{1ex}\\
&& + {~} 2\beta^2\Exps{\xi_t}{\norms{\nabla{f_{\xi_t}}(x_t) - \nabla f(x_t)}^2} \vspace{1ex}\\
& \leq & (1-\beta)^2\norms{v_{t-1} - \nabla f(x_{t-1})}^2 + 2(1-\beta)^2L^2\norms{x_t - x_{t-1}}^2 + 2\beta^2\sigma^2.
\end{array}
\end{equation*}
Taking the full expectation over the full history $\Fc_{t}$ of \eqref{eq:variance_redu}, and noticing that for $\beta \in (0, 1)$, $\frac{1 - (1-\beta)^{2t}}{1 - (1-\beta)^{2}} \leq \frac{1}{\beta}$, by induction, we can show that
\begin{equation*}
\arraycolsep=0.2em
\begin{array}{lcl}
\Exp{\norms{v_t - \nabla f(x_t)}^2}   & \leq & (1-\beta)^{2t}\Exp{\norms{v_0 - \nabla f(x_0)}^2} + 2\beta^2\sigma^2\frac{1 - (1-\beta)^{2t}}{1 - (1-\beta)^{2}}\vspace{1ex}\\
&& + {~} 2L^2\sum_{i=0}^{t-1}(1-\beta)^{2(t-i)}\Exp{\norms{x_{i+1} - x_i}^2} \vspace{1ex}\\
& \leq & (1-\beta)^{2t}\Exp{\norms{v_0 - \nabla f(x_0)}^2}  + 2\beta\sigma^2 \vspace{1ex}\\
&& + {~} 2L^2\sum_{i=0}^{t-1}(1-\beta)^{2(t-i)}\Exp{\norms{x_{i+1} - x_i}^2}.
\end{array}
\end{equation*}
This  proves \eqref{eq:variance_redu_full}.
\end{proof}

Next, we prove another property of our composite function $F$ in \eqref{eq:ncvx_prob}.
\begin{lemma}\label{lm:one-step_obj}
Let $\set{x_t}$ be generated by Algorithm~\ref{alg:A1} for solving \eqref{eq:ncvx_prob} and $G_{\eta}$ be defined by \eqref{eq:grad_mapping}.
Then, under Assumption~\ref{as:A1}, we have
\begin{equation}\label{eq:one-step_obj}
\begin{array}{lcl}
\Exp{F(x_{t+1}) - F^{\star}} & \leq & \Exp{F(x_t) - F^{\star}} - \left(\frac{1}{2\eta} - \frac{L}{2}\right)\Exp{\norms{x_{t+1} - x_t}^2} - \frac{\eta}{2}\Exp{\norms{G_{\eta}(x_t)}^2} \vspace{1ex}\\
&& + {~} \frac{\eta}{2}\Exp{\norms{\nabla f(x_t) - v_t}^2}.
\end{array}
\end{equation}
\end{lemma}

\begin{proof}
Let us denote by $\bar{x}_t := \prox_{\eta\psi}(x_t - \eta \nabla f(x_t))$.
From the optimality condition of this proximal operator, we have
\begin{equation*}
\iprod{\nabla f(x_t), \bar{x}_t - x_t} + \frac{1}{2\eta}\norms{\bar{x}_t - x_t}^2 + \psi(\bar{x}_t) \leq \psi(x_t) - \frac{1}{2\eta}\norms{x_t - \bar{x}_t}^2.
\end{equation*}
Similarly, from $x_{t+1} = \prox_{\eta\psi}(x_t - \eta v_t)$, we also have
\begin{equation*}
\iprod{v_t, x_{t+1} - x_t} + \frac{1}{2\eta}\norms{x_{t+1} - x_t}^2 + \psi(x_{t+1}) \leq \iprod{v_t, \bar{x}_t - x_t} + \frac{1}{2\eta}\norms{\bar{x}_t - x_t}^2 + \psi(\bar{x}_t) - \frac{1}{2\eta}\norms{\bar{x}_t - x_{t+1}}^2.
\end{equation*}
Combining the last two inequalities, we can show that
\begin{equation}\label{eq:key_eq1}
\begin{array}{lcl}
\psi(x_{t+1}) + \frac{1}{2\eta}\norms{x_{t+1} - x_t}^2 & \leq & \psi(x_t) - \frac{\eta}{2}\norms{G_{\eta}(x_t)}^2  - \frac{1}{2\eta}\norms{\bar{x}_t - x_{t+1}}^2\vspace{1ex}\\
&& + {~} \iprod{v_t, \bar{x}_t - x_{t+1}} - \iprod{\nabla f(x_t), \bar{x}_t - x_{t}}.
\end{array}
\end{equation}
By the Cauchy-Schwarz inequality, for any $\eta > 0$, we easily get
\begin{equation}\label{eq:key_eq2}
\iprod{\nabla f(x_t) - v_t, x_{t+1} - \bar{x}_t} \leq \frac{\eta}{2}\norms{\nabla f(x_t) - v_t}^2 + \frac{1}{2\eta}\norms{x_{t+1} - \bar{x}_t}^2.
\end{equation}
Finally, using the $L$-average smoothness of $f$, we can derive
\begin{equation*}
\arraycolsep=0.2em
\begin{array}{lcl}
f(x_{t+1}) + \psi(x_{t+1})  &\leq & f(x_t) + \iprod{\nabla f(x_t), x_{t+1} - x_t} + \frac{L}{2}\norms{x_{t+1} - x_t}^2 + \psi(x_{t+1}) \vspace{1ex}\\
& = &  f(x_t) - (\frac{1}{2\eta} - \frac{L}{2})\norms{x_{t+1} - x_t}^2 + \iprod{\nabla f(x_t), x_{t+1} - x_t} \vspace{1ex}\\
&& + {~} \psi(x_{t+1}) + \frac{1}{2\eta}\norms{x_{t+1} - x_t}^2 \vspace{1ex}\\
& \overset{\eqref{eq:key_eq1}}{\leq}  & f(x_t) - (\frac{1}{2\eta} - \frac{L}{2})\norms{x_{t+1} - x_t}^2 + \psi(x_t) + \iprod{\nabla f(x_t) - v_t, x_{t+1} - \bar{x}_t} \vspace{1ex}\\
&&  - {~} \frac{\eta}{2}\norms{G_{\eta}(x_t)}^2  - \frac{1}{2\eta}\norms{\bar{x}_t - x_{t+1}}^2\\
& \overset{\eqref{eq:key_eq2}}{\leq} &  f(x_t) + \psi(x_t) - (\frac{1}{2\eta} - \frac{L}{2})\norms{x_{t+1} - x_t}^2  + \frac{\eta}{2}\norms{\nabla f(x_t) - v_t}^2 - \frac{\eta}{2}\norms{G_{\eta}(x_t)}^2.\\
\end{array}
\end{equation*}
Taking the full expectation of both sides of the last inequality and noting that $F = f + \psi$, we obtain \eqref{eq:one-step_obj}.
\end{proof}

Now, we are ready to prove our main result, Theorem~\ref{thm:converge_rate} above.

\begin{proof}[\textbf{The proof of Theorem~\ref{thm:converge_rate}}]
First, summing up \eqref{eq:variance_redu_full} from $t :=0$ to $t := T$, we get
\begin{equation}\label{eq:variance_redu_sum}
\arraycolsep=0.2em
\begin{array}{lcl}
\sum_{t = 0}^{T}\Exp{\norms{v_t - \nabla f(x_t)}^2}  & \leq & \sum_{t = 0}^{T}(1-\beta)^{2t}\norms{v_0 - \nabla f(x_0)}^2  + 2(T+1)\beta\sigma^2 \vspace{1ex}\\
&& + {~} 2L^2\sum_{t = 0}^{T}\sum_{i=0}^{t-1}(1-\beta)^{2(t-i)}\Exp{\norms{x_{i+1} - x_i}^2} \vspace{1ex}\\
& \leq & \frac{1}{\beta}\norms{v_0 - \nabla f(x_0)}^2  + 2(T + 1)\beta\sigma^2 \vspace{1ex}\\
&& + {~} 2L^2\sum_{i = 0}^{T-1}\sum_{t=i+1}^{T}(1-\beta)^{2(t-i)}\Exp{\norms{x_{i+1} - x_i}^2} \vspace{1ex}\\
& \leq & \frac{1}{\beta}\norms{v_0 - \nabla f(x_0)}^2  + 2(T+1)\beta\sigma^2. \vspace{1ex}\\
&& + {~} 2L^2\sum_{i = 0}^{T-1}\frac{1}{\beta}\Exp{\norms{x_{i+1} - x_i}^2}.
\end{array}
\end{equation}
Next, summing up \eqref{eq:one-step_obj} from $t := 0$ to $t := T$, we obtain
\begin{equation*}
\arraycolsep=0.2em
\begin{array}{lcl}
\Exp{F(x_{T+1}) - F^{\star}}  & \leq & \left[ F(x_0) - F^{\star}\right] - \frac{\eta}{2}\sum_{t=0}^{T}\Exp{\norms{G_{\eta}(x_t)}^2} - \sum_{t=0}^{T}\left(\frac{1}{2\eta} - \frac{L}{2}\right)\Exp{\norms{x_{t+1} - x_t}^2} \vspace{1ex}\\
&& + {~} \frac{\eta}{2}\sum_{t=0}^{T}\Exp{\norms{v_t - \nabla f(x_t)}^2} \vspace{1ex}\\
& \overset{\eqref{eq:variance_redu_sum}}{\leq} &  \left[ F(x_0) - F^{\star}\right]  - \frac{\eta}{2}\sum_{t=0}^{T}\Exp{\norms{G_{\eta}(x_t)}^2} - \sum_{t=0}^{T}\left(\frac{1}{2\eta} - \frac{L}{2}\right)\Exp{\norms{x_{t+1} - x_t}^2} \vspace{1ex}\\
&& + {~} \frac{\eta}{2\beta}\Exp{\norms{v_0 - \nabla f(x_0)}^2} + \sum_{i = 0}^{T-1}\frac{L^2\eta}{\beta}\Exp{\norms{x_{i+1} - x_i}^2} + (T+1)\eta\beta\sigma^2.
\end{array}
\end{equation*}
Since $\eta \in \left(0, \frac{1}{2L}\right)$, we have $0 <  \frac{2L^2\eta^2}{1 - L\eta} < 1$.
Suppose $\frac{1}{2\eta} - \frac{L}{2} \geq \frac{L^2\eta}{\beta}$, i.e., $\beta \geq \frac{2L^2\eta^2}{1 - L\eta}$,  we have
\begin{equation*}
\Exp{F(x_{T+1}) - F^{\star}} \leq \left[F(x_0) - F^{\star}\right] - \frac{\eta}{2}\sum_{t=0}^{T}\Exp{\norms{G_{\eta}(x_t)}^2} 
+ \frac{\eta}{2\beta}\Exp{\norms{v_0 - \nabla f(x_0)}^2} + (T+1)\eta\beta\sigma^2,
\end{equation*}
which leads to \eqref{eq:thm_converge_rate}.

Now, if we choose $\eta := \frac{1}{2L(T+1)^{1/3}}$ and $\beta := \frac{1}{(T+1)^{2/3}}$, then we can verify that $\beta \geq \frac{2L^2\eta^2}{1-L\eta}$.
Moreover, \eqref{eq:thm_converge_rate} becomes
\begin{equation*} 
\frac{1}{T+1}\sum_{t=0}^{T}\Exp{\norms{G_{\eta}(x_t)}^2} \leq \frac{4L}{(T+1)^{2/3}}[F(x_0) - F^{\star}]
+ \frac{2\sigma^2}{(T+1)^{2/3}} + \frac{\Exp{\norms{v_0 - \nabla f(x_0)}^2}}{(T+1)^{1/3}}.
\end{equation*}
By Step~\ref{step:o2} of Algorithm~\ref{alg:A1} and the choice $\tilde{b} := \left\lceil \frac{(T+1)^{1/3}}{2} \right\rceil$, we have $\Exp{\norms{v_0 - \nabla{f}(x_0)}^2} \leq \frac{\sigma^2}{\tilde{b}} \leq \frac{2\sigma^2}{(T+1)^{1/3}}$.
Substituting this bound into the previous one and using $\Exp{\norms{G_{\eta}(\overline{x}_T)}^2} = \frac{1}{T+1}\sum_{t=0}^{T}\Exp{\norms{G_{\eta}(x_t)}^2}$, we obtain \eqref{eq:thm_opt_converge_rate}.

Finally, from \eqref{eq:thm_opt_converge_rate}, to guarantee $\Exp{\norms{G_{\eta}(\overline{x}_T)}^2} \leq \varepsilon^2$, we have $T+1 \geq \frac{\Delta_0^{3/2}}{\varepsilon^3}$, where $\Delta_0 := 4L[F(x_0) - F^{\star}] + 4\sigma^2$.
We can take $T := \left\lceil \frac{\Delta_0^{3/2}}{\varepsilon^3}\right\rceil$.
Therefore, the number of stochastic gradient evaluation is $\mathcal{T}_{\nabla{f}} = \tilde{b} + 2T =  \frac{\Delta_0^{1/2}}{2\varepsilon} + \frac{2\Delta_0^{3/2}}{\varepsilon^3}$.
Rounding it, we obtain $\mathcal{T}_{\nabla{f}}  = \left\lceil \frac{\Delta_0^{1/2}}{2\varepsilon} + \frac{2\Delta_0^{3/2}}{\varepsilon^3} \right\rceil$.
\end{proof}

\beforesec
\section{Concluding Remarks and Outlook}\label{sec:conc_remark}
\aftersec
Theorem~\ref{thm:converge_rate} only analyzes  a simple variant of Algorithm~\ref{alg:A1} with constant step-size $\eta = \BigO{\frac{1}{T^{1/3}}}$ and constant weight $\beta = \BigO{\frac{1}{T^{2/3}}}$.
It also uses a large initial mini-batch of size $\tilde{b} = \BigO{T^{1/3}}$.
Compared to SARAH-based methods, e.g., in \cite{fang2018spider,nguyen2017sarah,Pham2019}, Algorithm~\ref{alg:A1} is simpler since it is single-loop.
At each iteration, it uses only two samples compared to three ones in  \cite{Tran-Dinh2019a}.
We remark that the convergence of Algorithm~\ref{alg:A1} can be established by means of Lyapunov function as in  \cite{Tran-Dinh2019a}.

The result of this note can be extended into different directions:
\begin{compactitem}
\item We can also adapt our analysis to mini-batch, adaptive step-size $\eta_t$, and adaptive weight $\beta_t$ variants as in \cite{TranDinh2020f}.
If we use adaptive weight $\beta_t$ as in \cite{TranDinh2020f}, then we can remove the initial batch $\tilde{b}$ at Step~\ref{step:o2} of Algorithm~\ref{alg:A1}.
However, the convergence rate in Theorem~\ref{thm:converge_rate} will be $\BigO{\frac{\log(T)}{T^{2/3}}}$ instead of $\BigO{\frac{1}{T^{2/3}}}$.
The rate $\BigO{\frac{\log(T)}{T^{2/3}}}$ matches the result of \cite{Cutkosky2019} without bounded gradient assumption.

\item Our results, especially, Lemma~\ref{lm:variance_redu}, here can be applied to develop stochastic algorithms for solving other optimization problems such as compositional nonconvex optimization,  minimax problems, and reinforcement learning.

\item The idea here can also be extended to develop second-order methods such as sub-sampled and sketching Newton or cubic regularization-based methods.
\end{compactitem}
It is also interesting to incorporate this idea with adaptive schemes as done in \cite{Cutkosky2019} by developing different strategies such as curvature aid or quasi-Newton methods.

\bibliographystyle{plain}

\begin{thebibliography}{1}

\bibitem{arjevani2019lower}
Y.~Arjevani, Y.~Carmon, J.~C. Duchi, D.~J. Foster, N.~Srebro, and B.~Woodworth.
\newblock Lower bounds for non-convex stochastic optimization.
\newblock {\em arXiv preprint arXiv:1912.02365}, 2019.

\bibitem{Cutkosky2019}
A.~Cutkosky and F.~Orabona.
\newblock Momentum-based variance reduction in non-convex {SGD}.
\newblock In {\em Advances in Neural Information Processing Systems}, pages
  15210--15219, 2019.

\bibitem{fang2018spider}
C.~Fang, C.~J. Li, Z.~Lin, and T.~Zhang.
\newblock {SPIDER}: {N}ear-optimal non-convex optimization via stochastic path
  integrated differential estimator.
\newblock In {\em Advances in Neural Information Processing Systems}, pages
  689--699, 2018.

\bibitem{nguyen2017sarah}
L.~M. Nguyen, J.~Liu, K.~Scheinberg, and M.~Tak{\'a}{\v{c}}.
\newblock {SARAH}: {A} novel method for machine learning problems using
  stochastic recursive gradient.
\newblock {\em ICML}, 2017.

\bibitem{Pham2019}
H.~N. Pham, M.~L. Nguyen, T.~D. Phan, and Q.~Tran-Dinh.
\newblock {ProxSARAH}: {A}n efficient algorithmic framework for stochastic
  composite nonconvex optimization.
\newblock {\em J. Mach. Learn. Res.}, 21:1--48, 2020.

\bibitem{TranDinh2020f}
Q.~Tran-Dinh, D.~Liu, and L.~M. Nguyen.
\newblock Hybrid variance-reduced {SGD} algorithms for nonconvex-concave
  minimax problems.
\newblock {\em Tech. Report STOR.05.20, UNC-Chapel Hill (\href{arXiv preprint arXiv:2006.15266}{arXiv preprint arXiv:2006.15266})}, 2020.

\bibitem{Tran-Dinh2019a}
Q.~Tran-Dinh, N.~H. Pham, D.~T. Phan, and L.~M. Nguyen.
\newblock A hybrid stochastic optimization framework for stochastic composite
  nonconvex optimization.
\newblock {\em \href{https://arxiv.org/abs/1907.03793}{arXiv preprint arXiv:1907.03793}}, pages 1--49, 2019.

\end{thebibliography}

\vspace{3ex}
\noindent 
\textbf{Authors' information:}\\
\noindent 
Deyi Liu and Quoc Tran-Dinh$^{*}$\\
Department of Statistics and Operations Research\\
The University of North Carolina at Chapel Hill\\
Chapel Hill, NC 27599 \\
\textit{Email:} \texttt{deyi@live.unc.edu,quoctd@email.unc.edu} \vspace{0.5ex}\\
$^{*}$\textit{Corresponding author}.

\vspace{1ex}
\noindent 
Lam M. Nguyen, IBM Research, Thomas J. Watson Research Center, NY10598\\
\textit{Email:} \texttt{lamnguyen.mltd@ibm.com}

\end{document}